\documentclass[12pt,a4paper]{amsart}
\usepackage[utf8]{inputenc}
\usepackage{amsmath}
\usepackage{amsfonts}
\usepackage{amssymb}
\usepackage{tikz}
\usepackage{mathtools}
\usepackage[margin=2.0cm]{geometry}
\usepackage{wasysym}
\usepackage[all]{xy}
\usepackage{multicol}
\usepackage{tensor}

\newcommand{\C}{\mathbb{C}}
\newcommand{\Gg}{\mathcal{G}}
\newcommand{\Gu}{\mathcal{G}^{(0)}}
\newcommand{\Hh}{\mathcal{H}}
\newcommand{\Kk}{\mathcal{K}}
\newcommand{\Hu}{\mathcal{H}^{(0)}}
\newcommand{\Uu}{\mathcal{U}}

\newcommand{\inv}{^{-1}}
\DeclareMathOperator{\End}{End}
\newcommand{\N}{\mathbb{N}}
\newcommand{\restimes}[2]{\tensor[_{#1}]{\times}{_{#2}}}

\DeclareMathOperator{\supp}{\textrm{supp}}

\DeclareMathOperator{\Ind}{Ind}
\DeclareMathOperator{\range}{range}
\DeclareMathOperator{\Aut}{Aut}
\DeclareMathOperator{\Bis}{B}

\newtheorem{theorem}{Theorem}[section]
\newtheorem{proposition}[theorem]{Proposition}
\newtheorem{lemma}[theorem]{Lemma}
\newtheorem{corollary}[theorem]{Corollary}
\theoremstyle{definition}
\newtheorem{definition}[theorem]{Definition}
\theoremstyle{remark}
\newtheorem{remark}[theorem]{Remark}
\newtheorem{example}[theorem]{Example}
\numberwithin{equation}{section}

\author[Brownlowe]{Nathan Brownlowe}
\address{Nathan Brownlowe, David Pask and David Robertson \\ School of Mathematics and
Applied Statistics  \\
University of Wollongong\\
NSW  2522\\
AUSTRALIA} 
\email{nathanb@uow.edu.au,david\_pask@uow.edu.au,dave84robertson@gmail.com}
\address{Jacqui Ramagge \\ School of Mathematics and Statistics \\ University of Sydney \\ NSW 2006 \\ AUSTRALIA}
\email{jramagge@gmp18.hbs.edu}
\address{Michael F. 
Whittaker \\ School of Mathematics and Statistics  \\
University of Glasgow\\ University Gardens \\ Glasgow Q12 8QW \\ SCOTLAND} 
\email{mfwhittaker@gmail.com}
\author[Pask]{David Pask}
\author[Ramagge]{Jacqui Ramagge}
\author[Robertson]{David Robertson}
\author[Whittaker]{Michael F. Whittaker}

\title{Zappa-Sz\'ep product groupoids and $C^*$-blends}

\thanks{This research was supported by the Australian Research Council and the University of Wollongong through a University Research Committee Grant to the fourth author. We would also like to take the opportunity to thank Ruy Exel and Charles Starling for many interesting discussions around the topics in this article.}


\begin{document}
\maketitle

\begin{abstract}
We study the external and internal Zappa-Sz\'ep product of 
topological groupoids. We show that under natural continuity assumptions the 
Zappa-Sz\'ep product groupoid is \'etale if and only if the individual 
groupoids are \'etale. In our main result we show that the 
$C^*$-algebra of a 
locally compact Hausdorff \'etale Zappa-Sz\'ep product groupoid is a 
$C^*$-blend, in the sense of Exel, of the individual groupoid 
$C^*$-algebras. 
We finish with some examples, including groupoids built from 
$*$-commuting 
endomorphisms, and skew product groupoids. 
\end{abstract}

\section{Introduction}

Group theory has for many years provided fertile ground for mathematicians 
working in $C^*$-algebras. Indeed, the 
notion of a group $C^*$-algebra is as old 
as the field itself, and possesses many interesting and natural properties. 
For 
instance, it is well known that the $C^*$-algebra of the direct product of 
groups $G$ and $H$ is the tensor product of the individual group 
$C^*$-algebras $C^*(G)$ and $C^*(H)$ (see, for instance, 
\cite[Examples~II.10.3.15]{Black}). It is also well known that the 
$C^*$-algebra of the semidirect product induced by an action 
$H\curvearrowright G$ is the crossed product $C^*$-algebra induced by the 
action $H\curvearrowright C^*(G)$ (see also 
\cite[Examples~II.10.3.15]{Black}). In this article we are interested in a 
third, and more general, notion of a product of groups called the 
Zappa-Sz\'{e}p product. As a consequence of our main result we are able to 
answer the natural questions: what is the group $C^*$-algebra of the 
Zappa-Sz\'{e}p product of two groups $G$ and $H$, and what does it have to do 
with $C^*(G)$ and $C^*(H)$?  


A Zappa-Sz\'{e}p product of groups $G$ and $H$ is a generalisation of a 
semidirect product, in the sense that neither group is necessarily normal in 
the product. Like the semidirect product, there is an internal and external 
Zappa-Sz\'ep product. 
A group $K$ is the internal Zappa-Sz\'{e}p product of $G$ and $H$ if 
$G,H$ are subgroups of $K$ such that $K=GH$ as a set, and $G \cap H = \{e\}$. 
It then follows that $K$ is in bijection with $G \times H$ as a set. 
If $G,H$ are both normal in $K$, then $K$ is isomorphic to the 
direct product of groups  $G \times H$ with pointwise multiplication. 
If only $G$ is normal, then $H$ acts 
on $G$ by conjugation $(h,g) \mapsto hgh^{-1}$, and $K$ is isomorphic to the 
semidirect product $G \rtimes H$. 
In general neither $G$ nor $H$ need be normal in $K$, so that direct and semidriect products are special cases of the Zappa-Sz\'ep product. 
In any case, since $K=GH$, given any $h\in H$ and $g\in G$ there are  elements 
$h\cdot g\in G$ and $h|_g\in H$ such that $hg=(h\cdot g)h|_g$, and the condition
$G \cap H = \{e\}$ forces $h\cdot g$ and $h|_g$ to be uniquely determined. 
The {\it action 
map} $\cdot: H\times G\to G$ is given by $(h,g)\mapsto h\cdot g$, and {\it the 
restriction map} $\left| \right. : H \times G \to H$ is given by $(h,g) 
\mapsto h |_g$. These maps can be used to define an associative multiplication 
and inversion on $G \times H$ by
\[
(g,h)(g',h') = (g(h\cdot g'),h|_{g'}h')\quad\text{ and }\quad (g,h)\inv = 
(h\inv\cdot g\inv, h\inv|_{g\inv}),
\]
respectively. The resulting group $G\Join H$ is called the \emph{Zappa-Sz\'ep 
product} of $G$ and $H$ and there is an isomorphism $G \Join H  \cong K$ given 
by $(g,h) \mapsto gh$ (see \cite{Z,Szep1,Szep2,Szep3}).
Given an arbitrary pair of groups $G$ and $H$ with a left action of $H$ on $G$ and a right action of $G$ on $H$ 
we  construct the external Zappa-Sz\'ep product $G\Join H$ as the set $G \times H$
with a product defined analogously.

In this article we extend the definitions given above to the more general 
setting of groupoids. Groupoids 
have featured prominently in the study of 
$C^*$-algebras since the seminal work of Renault~\cite{Ren}. The diversity of 
examples of groupoid $C^*$-algebras has been a feature of their success, 
especially in the study of $C^*$-algebras associated to dynamical systems. 
Building on the recent work on Zapp-Sz\'ep product semigroups \cite{BRRW} 
(which was influenced by the work on self-similar actions in 
\cite{nek_book,Law,LRRW,Wal}), 
we introduce 
the Zappa-Sz\'{e}p product of 
topological groupoids. Our notion of the Zappa-Sz\'{e}p product of groupoids 
is not new on an 
algebraic level. In \cite{AA} Aguiar and Andruskiewitsch introduced the notion 
of a matched pair of groupoids; a pair of groupoids $(\Gg,\Hh)$ is a 
matched pair if and only if there is a well-defined Zappa-Sz\'{e}p product 
$\Gg\Join\Hh$. (Note that in \cite{AA} $\triangleleft$ corresponds to our 
restriction, and 
$\triangleright$ corresponds to our action.) More generally, a Zappa-Sz\'{e}p 
product of categories was introduced in the work of Brin \cite{Brin}. However, 
in the specific case of groupoids, in which every morphism is an isomorphism, 
Brin's axioms (P1)--(P8) (see \cite[p.406]{Brin}) can be simplified.

After formalising the Zappa-Sz\'ep product $\Gg \Join \Hh$ of 
topological groupoids, we switch focus to $C^*$-algebras, and consider the 
relationship between the groupoid $C^*$-algebras of $\Gg$, $\Hh$ and 
$\Gg\Join\Hh$. In our main theorem we prove that $C^*(\Gg\Join\Hh)$ is a 
$C^*$-blend of $C^*(\Gg)$ and $C^*(\Hh)$. $C^*$-blends were recently 
introduced by Exel in \cite{Exel2014} during his examination of the possible 
algebraic and $C^*$-algebraic structures one can put on the tensor product $A 
\otimes B$ of two algebras $A$ and $B$. While several examples are examined in 
\cite{Exel2014}, any new theory in $C^*$-algebras always benefits from additional examples. Our main theorem allows us to work at the level 
of 
groupoids to describe concrete examples of $C^*$-blends, rather than with their $C^*$-algebraic completions. We are thus able to 
describe several 
new examples, including those involving Deaconu-Renault groupoids, and 
skew product 
groupoids. 

While we were at first surprised to learn that $C^*(\Gg\Join\Hh)$ is a 
$C^*$-blend of 
the individual groupoid $C^*$-algebras, once you scratch beneath the surface, 
the answer is a natural one. As discussed in \cite{Exel2014}, $C^*$-blends 
generalise crossed products of 
$C^*$-algebras by groups, in the sense that $A\rtimes G$ is a $C^*$-blend of 
$A$ and $C^*(G)$. The theory of $C^*$-blends then feels like a natural home 
for the $C^*$-algebras of Zappa-Sz\'ep products of groups, given that they 
generalise semidirect products, whose $C^*$-algebras are crossed products. 
There is also an interesting similarity between the conditions under which the 
product of two groups is a Zappa-Sz\'ep product, and when the product of two 
$C^*$-algebras is a $C^*$-blend. If $G$ and $H$ are subgroups 
of a 
group $K$, then $GH=\{gh:g\in G,h\in H\}$ is a group if and only if $GH=HG$, 
and is isomorphic to $G\Join H$ if and only if $G\cap H=\{e\}$. 
If $A,B$ are $C^*$-subalgebras of a $C^*$-algebra $C$, it is not hard to show 
that the statements $AB=\overline{\text{span}}\{ ab : a\in A, b\in B \}$ a 
$C^*$-algebra, $AB=BA$, and $AB$ a $C^*$-blend of 
$A$ and $B$ are equivalent. (See Remark~\ref{ref: chars of blends}.)

This article is organised as follows.
Section \ref{sec:prelim} provides preliminaries on topological groupoids and 
gives three specific examples that will be used heavily throughout the paper: 
transformation groupoids, Deaconu-Renault groupoids, and skew product 
groupoids. In Section \ref{sec:Zappa-Szep} we discuss the Zappa-Sz\'{e}p 
product $\Gg \Join \Hh$ of two groupoids $\Gg$ and $\Hh$, and show how the so 
called ``arrow space" of each groupoid along with certain identifications 
describes the ``arrow space" of $\Gg \Join \Hh$. We then provide an internal 
characterisation of a Zappa-Sz\'{e}p groupoid and show that if $\Gg$ and $\Hh$ 
are \'{e}tale groupoids, then so is $\Gg \Join \Hh$. In Section~\ref{sec: C* 
section} we prove our main theorem, which says that $C^*(\Gg\Join\Hh)$ is a 
$C^*$-blend of $C^*(\Gg)$ and $C^*(\Hh)$. We finish in 
Section~\ref{sec:examples} by examining several examples of Zappa-Sz\'{e}p 
product groupoids and their $C^*$-algebras.

\section{Preliminaries}\label{sec:prelim}

Let $\Gg$ be a set and suppose $\Gg^{(2)} \subset \Gg \times \Gg$. We say $\Gg$ is a \emph{groupoid} if there is a multiplication $(g,h) \mapsto gh$ from $\Gg^{(2)}$ to $\Gg$ and an inverse map $g \mapsto g\inv$ from $\Gg$ to $\Gg$ satisfying the following.
\begin{enumerate}
 \item If $(g,h), (h,k) \in \Gg^{(2)}$ then $(gh,k), (g,hk)\in\Gg^{(2)}$ and $g(hk) = (gh)k$.
 \item We have $(g\inv)\inv = g$ for all $g\in\Gg$.
 \item We have $(g,g\inv)\in \Gg^{(2)}$ for all $g\in\Gg$ and if $(g,h)\in\Gg^{(2)}$, then $g\inv(gh) = h$ and $(gh)h\inv = g$.
\end{enumerate}
We call $\Gg^{(2)}$ the set of \emph{composable pairs} and $\Gg^{(0)} := \{ gg\inv : g\in \Gg \}$ the set of \emph{units}. The \emph{range} map $\Gg \to \Gg^{(0)}$ is given by $g\mapsto gg\inv$ and the \emph{source} map $\Gg \to \Gg^{(0)}$ is given by $g \mapsto g\inv g$.

A useful interpretation of this definition is that an element $g$ of a groupoid $\Gg$ is an arrow pointing from the source of $g$ to the range of $g$.
\[
\begin{tikzpicture}
\node (l) at (0,0){$gg\inv$};
\node (r) at (4,0){$g\inv g$};
\draw [->] (r) -- node[auto] {$g$} (l) ;
\end{tikzpicture}
\]
We think of inversion as reversing the direction of the arrow, and a pair $(g,h)\in\Gg^{(2)}$ whenever the source of $g$ agrees with the range of $h$ in $\Gg^{(0)}$; the product $gh$ is then the composition of arrows.

We call $\Gg$ a \emph{topological groupoid} if $\Gg$ is a topological space and the multiplication and inversion maps are continuous, where $\Gg^{(2)}$ has the relative product topology. We call a topological groupoid \emph{\'etale} if the range (and hence also the source) map is a local homeomorphism.

\begin{example}[Tranformation groupoids {\cite[p.8]{Pat}}]
Let $G$ be a topological group acting continuously on a topological space $X$. There is a groupoid $G\ltimes X$ given by
\begin{align*}
 &G \ltimes X = G \times X \\
 &(G \ltimes X)^{(2)} = \{((g,y),(h,x)) : g,h\in G, x,y\in X, \text{ and } y=h\cdot x \} \\
 &(g,h\cdot x)(h,x) = (gh,x) \\
 &(g,x)\inv = (g\inv,g\cdot x)
\end{align*}
We call $G\ltimes X$ a \emph{transformation groupoid}. The unit space is given by $\{e\} \times X \cong X$ and an element $(g,x)$ has range $g\cdot x$ and source $x$. When given the product topology $G\ltimes X$ is a topological groupoid, and is \'etale if and only if $G$ has the discrete topology.
\end{example}

\begin{example}[Deaconu-Renault groupoids {\cite[Section 
3]{SimsWilliams2015}}]\label{eg: DR groupoid}
The study of topological groupoids associated with endomorphisms began with the seminal paper of Deaconu \cite{Deaconu1995}, and has been extended in various directions. In order to associate a groupoid $C^*$-algebra to a $k$-graph, Kumjian and Pask \cite{KP2000} began the extension of Deaconu-Renault groupoids to transformations of $\N^k$. More recently, a completely general description of groupoids associated with transformations of $\N^k$ and their topology is given in \cite[Section 3]{SimsWilliams2015}.

Let $X$ be a topological space and suppose $k\geq 1$. Following the literature \cite{AR} we say that $\sigma:X \to X$ is an endomorphism if $\sigma$ is a surjective local homeomorphism. Suppose $\theta:\N^k \to \End(X)$ is an action of $\N^k$ on $X$ by continuous endomorphisms. Define the \emph{Deaconu-Renault groupoid} $X \rtimes_\theta \N^k$ by
\begin{align*}
&X \rtimes_\theta \N^k = \{ (x,m-n,y) : x,y\in X, m,n \in \N^k, \theta_m(x) = \theta_n(y) \} \\
&(X \rtimes_\theta \N^k)^{(2)} = \{ ((x,k-l,y),(w,m-n,z)) : y=w \} \\
&(x,k-l,y)(y,m-n,z) = (x,(k+m)-(l+n),z) \\
&(x,m-n,y)\inv = (y,n-m,x) \\
\end{align*}
The unit space is $\{(x,0,x) : x\in X\} \cong X$ and an element $(x,m-n,y)$ has range $x$ and source $y$. With the topology described in \cite[Section 3]{SimsWilliams2015}, the groupoid $X \rtimes_\theta \N^k$ becomes a topological groupoid.
\end{example}

\begin{example}[Skew product groupoids {\cite[Section 4]{KQR2001}}] \label{ex:skewproductgroupoid}
Fix an \'etale groupoid $\Gg$, a discrete group $A$ and a continuous 
homomorphism $c : \Gg \to A$. The \emph{skew-product groupoid} $\Gg(c)$ 
is is given by
\begin{align*}
&\Gg(c)=\Gg \times A \\
&\Gg(c)^{(2)} = \{ (g,\alpha)(h,\beta) : (g,h)\in\Gg^{(2)}, \beta = 
 \alpha 
 c(g) \} \\
&(g,\alpha)(h,\alpha c(g)) = (gh,\alpha) \\
&(g,\alpha)\inv = (g\inv, \alpha c(g)).
\end{align*}
We will denote\footnote{This unusual choice of labelling the range and 
source 
maps by $b$ and $t$ is explained in Section~\ref{sec:Zappa-Szep}.} the range and source 
maps by 
$b,t : \Gg(c) \to \Gg(c)^{(0)}$ 
respectively. We have
\begin{align*}
 b(g,\alpha) &= (g,\alpha)(g\inv,\alpha c(g)) = (gg\inv,\alpha) \\
 t(g,\alpha) &= (g\inv,\alpha c(g))(g,\alpha) = (g\inv g, \alpha c(g)).
\end{align*}
The unit space is therefore $\Gg(c)^{(0)} = \Gg^{(0)} \times A$.
\end{example}

\section{Zappa-Sz\'ep products of groupoids}\label{sec:Zappa-Szep}

In this section we will describe the Zappa-Sz\'{e}p product of  two groupoids with bijective unit spaces. To do this we need to recall a fibre product of two sets. Suppose $G$, $H$, and $X$ are sets such that $\gamma:G \to X$ and $\eta:H \to X$ are maps. The \emph{fibre product} (or pull-back) of $G$ and $H$ over $X$ is the set
\[
G \restimes{\gamma}{\eta} H:=\{(g,h) \in G \times H : \gamma(g)=\eta(h)\}.
\]

Let $\Gg$ and $\Hh$ be groupoids and suppose there is a bijection $\Gg^{(0)} \to \Hh^{(0)}$. For simplicity, we will consider them the same set and write $\Gu =\Uu = \Hu$. As in \cite{AA} we will think of elements of $\Gg$ as vertical arrows and elements of $\Hh$ as horizontal arrows as shown below:
\[
\begin{tikzpicture}
\node (b) at (0,-1){$b(g)=gg\inv$};
\node (t) at (0,1){$t(g) = g\inv g$};
\node at (-0.3,0){$g$};
\draw [->] (t) -- (b);

\node (l) at (3,0){$l(h) = hh\inv$};
\node (r) at (7,0){$r(h) = h\inv h$};
\node at (5,0.3){$h$};
\draw [->] (r) -- (l);

\end{tikzpicture}
\]
We have denoted the range and source maps of $\Gg$ by $b$ for bottom and $t$ for top, respectively, and the range and source maps of $\Hh$ by $l$ for left and $r$ for right, respectively. Suppose that there are maps
\[
 \cdot : \Hh \restimes{r}{b} \Gg \to \Gg \text{ given by } (h,g) \mapsto h\cdot g 
 \quad\text{and}\quad
 |: \Hh \restimes{r}{b} \Gg \to \Hh \text{ given by } (h,g) \mapsto h|_g
\]
satisfying
\begin{multicols}{2}
\begin{itemize}
\item[(ZS1)] $(h_1h_2)\cdot g = h_1\cdot(h_2\cdot g)$
\item[(ZS2)] $h\cdot(g_1g_2) = (h\cdot g_1)(h|_{g_1} \cdot g_2)$
\item[(ZS3)] $h|_{g_1g_2} = (h|_{g_1})|_{g_2}$
\item[(ZS4)] $(h_1h_2)|_g = (h_1|_{h_2\cdot g})(h_2|_{g})$
\item[(ZS5)] $b(h\cdot g) = l(h)$
\item[(ZS6)] $r(h|_g) = t(g)$
\item[(ZS7)] $t(h\cdot g) = l(h|_g)$
\item[(ZS8)] $b(g)\cdot g = g$
\item[(ZS9)] $h|_{r(h)} = h$,
\end{itemize}
\end{multicols}

\noindent
whenever these formulae make sense. We call $\cdot$ the \emph{action} map and 
$|$ the \emph{restriction} map. These axioms appear in \cite{AA}, where the 
action is denoted $\triangleright$, and the restriction is denoted 
$\triangleleft$.

Before we can construct the Zappa-Sz\'ep product groupoid, we need the following lemma.

\begin{lemma}[cf.\ {\cite[Lemma 1.2]{AA}}] \label{lem:actionrestrictioninverses}
For any $(h,g) \in \Hh\restimes{r}{b} \Gg$ we have
\begin{enumerate}
 \item\label{old_ZS9} $h \cdot r(h) = l(h)$,
 \item\label{old_ZS10} $b(g)|_g = t(g)$,
 \item\label{inv_act} $(h\cdot g)\inv = h|_g\cdot g\inv$, and
 \item\label{inv_rest} $(h|_g)\inv = h\inv|_{h\cdot g}$.
\end{enumerate}
\end{lemma}

\begin{proof}
For \eqref{old_ZS9}, using (ZS2), (ZS9), and that $b(g)=r(h)$ we compute
\begin{equation*}
h \cdot g =h\cdot(b(g)g)=(h\cdot b(g))(h|_{b(g)}\cdot g)=(h \cdot r(h))(h\cdot g),
\end{equation*}
which implies $h \cdot r(h)=b(h \cdot g)=l(h)$ by (ZS5).

For \eqref{old_ZS10}, using (ZS4) and then (ZS8) we compute
\begin{equation*}
h|_g=(h r(h))|_g=h|_{(r(h)\cdot g)}r(h)|_g=h|_{b(g)\cdot g} b(g)|g=h|_g 
b(g)|_g,
\end{equation*}
which implies that $b(g)|_g=r(h|_g)=t(g)$ by (ZS6).

For \eqref{inv_act} and \eqref{inv_rest}, using (ZS5), part \eqref{old_ZS9}, and (ZS2) we have
\[
 (h\cdot g)(h\cdot g)\inv = b(h\cdot g)
 = l(h)
 = h\cdot r(h)
 = h\cdot b(g)
 = h\cdot(gg\inv)
 = (h\cdot g)(h|_g\cdot g\inv) 
\]
and \eqref{inv_act} follows by cancelling $(h\cdot g)$ on the left. Similarly, using (ZS6), part \eqref{old_ZS10}, and (ZS4) we have
\[
 (h|_g)\inv(h|_g) = r(h|_g) = t(g) = b(g)|_g = r(h)|_g = (h\inv h)|_g = (h\inv|_h\cdot g)(h|_g)
\]
and \eqref{inv_rest} follows by cancelling $(h|_g)$ on the right.
\end{proof}

\noindent
We define the \emph{Zappa-Sz\'ep product} as the set
\[
 \Gg\Join \Hh = \Gg \restimes{t}{l} \Hh,
\]
with the range of $(g,h) \in \Gg \Join \Hh$ given by $(b(g),b(g)) \in \Uu 
\times \Uu$, and the source of $(g,h)$ given by $(r(h),r(h)) \in \Uu \times 
\Uu$. We have
\[
 (\Gg\Join \Hh)^{(2)} = \{ ((g_1,h_1),(g_2,h_2)) : r(h_1) = b(g_2) \} .
\]
We define multiplication by
\[
 (g_1,h_1)(g_2,h_2) = (g_1(h_1\cdot g_2),h_1|_{g_2}h_2)
\]
and inversion by
\[
 (g,h)\inv = (h\inv\cdot g\inv, h\inv|_{g\inv}).
\]

\begin{proposition}\label{prop:ZS groupoid}
With the above structure, $\Gg\Join\Hh$ is a groupoid with unit space $\Gg^{(0)} \restimes{t}{l} \Hh^{(0)} \cong \Uu$.
\end{proposition}

\begin{proof}
Conditions (ZS5) and (ZS6) imply that the multiplication is well-defined and 
(ZS7) shows $\Gg\Join \Hh$ is closed under multiplication. Suppose 
$((g_1,h_1),(g_2,h_2))$ and $((g_2,h_2),(g_3,h_3)) \in (\Gg\Join\Hh)^{(2)}$. 
It is easy to see that $(g_1(h_1\cdot g_2),h_1|_{g_2}h_2),(g_3,h_3)) \in 
(\Gg\Join\Hh)^{(2)}$ and $((g_1,h_1),(g_2(h_2\cdot g_3),h_2|_{g_3}h_3)) \in 
(\Gg\Join\Hh)^{(2)}$. Some parenthetical gymnastics using (ZS1-4) shows that 
the following associativity holds:
\begin{equation*}
((g_1  (h_1\cdot g_2),h_1|_{g_2}h_2))(g_3,h_3) = (g_1,h_1)(g_2(h_2 \cdot g_3),h_2|_{g_3}h_3).
\end{equation*}
For $(g,h) \in \Gg\Join\Hh$ we have
\begin{align}
\notag
 ((g,h)\inv)\inv &= (h\inv\cdot g\inv, h\inv|_{g\inv})\inv \\
 \label{four_prod}
 &= ((h\inv|_{g\inv})\inv\cdot (h\inv\cdot g\inv)\inv, (h\inv|_{g\inv})\inv|_{(h\inv\cdot g\inv)\inv} ).
\end{align}
By Lemma \ref{lem:actionrestrictioninverses} \eqref{inv_act} the first term in \eqref{four_prod} becomes
\[
(h\inv|_{g\inv})\inv\cdot (h\inv\cdot g\inv)\inv = (h\inv|_{g\inv})\inv\cdot ( h\inv|_{g\inv} \cdot g ) = g
\]
and by Lemma \ref{lem:actionrestrictioninverses} \eqref{inv_rest} the second term in \eqref{four_prod} becomes
\[
 (h\inv|_{g\inv})\inv|_{(h\inv\cdot g\inv)\inv} = (h|_{(h\inv\cdot 
 g\inv)})|_{(h\inv\cdot g\inv)\inv} = h.
\]
So $((g,h)\inv)\inv = (g,h)$ as required. From (ZS5) we have 
$((g,h),(h\inv\cdot g\inv,h\inv|_{g\inv})) \in \Gg\Join\Hh^{(2)}$. Using 
(ZS1), (ZS3), (ZS8) and Lemma \ref{lem:actionrestrictioninverses} 
\eqref{old_ZS10} we have
\begin{align*}
 (g,h)(h\inv\cdot g\inv, h\inv|_{g\inv}) &= (g(h\cdot(h\inv\cdot g\inv)),h|_{h\inv\cdot g\inv}h\inv|_{g\inv}) \\
 &= (g(l(h)\cdot g\inv),l(h)|_{g\inv}) \\
 &= (g(b(g\inv)\cdot g\inv),b(g\inv)|_{g\inv}) \\
 &= (b(g),b(g)),
\end{align*}
so if $((g_1,h_1),(g_2,h_2))\in (\Gg\Join\Hh)^{(2)}$, then Lemma 
\ref{lem:actionrestrictioninverses} \eqref{old_ZS9} and (ZS9) imply
\[
 (g_1,h_1)(g_2,h_2)(g_2,h_2)\inv = (g_1,h_1)(b(g_2),b(g_2))=(g_1(h_1 \cdot r(h_1)),(h_1|_{r(h_1)}) r(h_1))= (g_1,h_1).
\]
A similar argument shows
\[
 (g,h)\inv (g,h) = (r(h),r(h))\quad\text{and}\quad 
 (g_1,h_1)\inv(g_1,h_1)(g_2,h_2) = (g_2,h_2),
\]
as required.
\end{proof}

\begin{remark}
In the proof of Proposition \ref{prop:ZS groupoid}, note that we used all of the 
axioms (ZS1--9). This shows the 
necessity of the rather large number of axioms.
\end{remark}

From now on we will freely identify $(\Gg\Join\Hh)^{(0)}$ with $\Uu$. We can 
consider $\Gg$ and $\Hh$ as subgroupids of $\Gg\Join\Hh$ via
\[
\Gg \cong \{ (g,t(g)) : g\in \Gg \} \quad \text{ and } \quad \Hh \cong \{ (l(h),h) : h\in \Hh \}.
\]

\noindent
Since $\Gg\Join \Hh = \Gg \restimes{t}{l} \Hh$ as a set, we may represent elements $(g',h')\in\Gg \Join \Hh$ and $(h,g) \in \Hh \restimes{r}{b} \Gg$ as pairs of arrows between elements of $\Uu$ as follows:
\[
\begin{tikzpicture}
\node (b) at (0,0){$b(g')$};
\node (t) at (0,2){$t(g') = l(h')$};
\draw [->] (t) -- node[auto] {$g'$}(b);
\node (r) at (3,2){$r(h')$};
\draw [->] (r) -- node[auto] {$h'$} (t);
\node at (5,1.0) {and};

\node (b) at (7,0){$l(h)$};
\node (t) at (10,0){$r(h) = b(g),$};
\draw [->] (t) -- node[auto] {$h$}(b);
\node (r) at (10,2){$t(g)$};
\draw [->] (r) -- node[auto] {$g$} (t);
\end{tikzpicture}
\]

\noindent
respectively.
In the Zappa-Sz\'{e}p product $\Gg\Join \Hh$, the element $(h,g) \in \Hh \restimes{r}{b} \Gg$ is identified with $(h\cdot g,h|_g) \in \Gg\Join \Hh$. Visually, this amounts to the following two diagrams being identified:
\[
\begin{tikzpicture}
\node (b) at (0,0){$l(h)$};
\node (t) at (3,0){$r(h) = b(g)$};
\draw [->] (t) -- node[auto] {$h$}(b);
\node (r) at (3,2){$t(g)$};
\draw [->] (r) -- node[auto] {$g$} (t);

\node (e) at (5.0,1.0){$\longleftrightarrow$};

\node (b) at (7,0){$b(h\cdot g)$};
\node (t) at (7,2){$t(h\cdot g) = l(h|_g)$};
\draw [->] (t) -- node[auto] {$h\cdot g$}(b);
\node (r) at (10,2){$r(h|_g)$};
\draw [->] (r) -- node[auto] {$h|_g$} (t);
\end{tikzpicture}
\]
This geometric identification corresponds to the algebraic axioms (ZS5--7). 
Now suppose 
$((g_1,h_1),(g_2,h_2)) \in (\Gg\Join\Hh)^{(2)}$, then the 
product $(g_1,h_1)(g_2,h_2) = (g_1(h_1\cdot g_2),h_1|_{g_2}h_2)$ is represented by
\[
\begin{tikzpicture}
\node (a) at (0,0){};
\node (b) at (0,1.5){};
\node (c) at (2,1.5){};
\node (d) at (2,3){};
\node (e) at (4,3){};
\draw [<-] (a) -- node[auto] {$g_1$} (b);
\draw [<-] (b) -- node[auto] {$h_1$} (c);
\draw [<-] (c) -- node[auto] {$g_2$} (d);
\draw [<-] (d) -- node[auto] {$h_2$} (e);

\node (eq) at (5,2) {$\longleftrightarrow$};

\node (a') at (7,0){};
\node (b') at (7,1.5){};
\node (c') at (7,3){};
\node (d') at (9,3){};
\node (e') at (11,3){};
\draw [<-] (a') -- node[auto] {$g_1$} (b');
\draw [<-] (b') -- node[auto] {$h_1\cdot g_2$} (c');
\draw [<-] (c') -- node[auto] {$h_1|_{g_2}$} (d');
\draw [<-] (d') -- node[auto] {$h_2$} (e');

\end{tikzpicture}
\]
Finally, if $(g,h)\in\Gg\Join\Hh$, then the inverse $(g,h)\inv = (h\inv\cdot g\inv, h\inv|_{g\inv})$ is represented by
\[
\begin{tikzpicture}

\node (a) at (0,0){};
\node (b) at (0,1.5){};
\node (c) at (2,1.5){};
\draw [<-] (a) -- node[auto] {$g$} (b);
\draw [<-] (b) -- node[auto] {$h$} (c);

\node (m) at (2.5,1){$\longmapsto$};

\node (d) at (3.5,0){};
\node (e) at (5.5,0){};
\node (f) at (5.5,1.5){};
\draw [<-] (d) -- node[below] {$h\inv$} (e);
\draw [<-] (e) -- node[auto] {$g\inv$} (f);

\node (eq) at (6.5,1) {$\longleftrightarrow$};

\node (g) at (9,0){};
\node (h) at (9,1.5){};
\node (i) at (11,1.5){};
\draw [<-] (g) -- node[auto] {$h\inv\cdot g\inv$} (h);
\draw [<-] (h) -- node[auto] {$h\inv|_{g\inv}$} (i);

\end{tikzpicture}
\]

\noindent
The following proposition determines when a given groupoid decomposes as a Zappa-Sz\'ep product.

\begin{proposition}[Internal Zappa-Sz\'{e}p products] 
\label{prop:internalproduct}
Let $\Kk$ be a groupoid and let $\Gg$ and $\Hh$ be subgroupoids. Suppose that for any $k\in \Kk$ there is a unique pair $(g,h) \in (\Gg \times \Hh) \cap \Kk^{(2)}$ such that $k = gh$. Then $\Kk \cong \Gg \Join \Hh$.
\end{proposition}

\begin{proof}
We must first show that $\Gg$ and $\Hh$ admit the structure required to take a Zappa-Sz\'ep product. First, fix $u\in \Kk^{(0)}$. Then there is a unique $g\in \Gg$ and $h \in \Hh$ with $u = gh$. But then $u = gg\inv \in \Gg$ and $u = h\inv h \in \Hh$. Since $u = uu$, uniqueness forces $h = u = g$. Hence $\Kk^{(0)} = \Gg^{(0)} = \Hh^{(0)}$. We define action and restriction maps using the unique decomposition; given a pair $(g,h) \in (\Gg\times\Hh)\cap\Kk^{(2)}$ let $(g\cdot h, g|_h) \in (\Gg\times\Hh)\cap\Kk^{(2)}$ be the unique pair such that
\[
 gh = (g\cdot h)(g|_h).
\]
Routine calculations show these maps satisfy (ZS1--9), and $k \mapsto (g,h)$, 
where $k = gh$, is an isomorphism $\Kk \cong \Gg \Join \Hh$.
\end{proof}

Using Proposition \ref{prop:internalproduct} and Lemma \ref{lem:actionrestrictioninverses} we can show that taking groupoid Zappa-Sz\'ep products is symmetric.

\begin{corollary}
Any groupoid Zappa-Sz\'ep product $\Gg\Join\Hh$ is isomorphic to a Zappa-Sz\'ep product $\Hh\Join\Gg$.
\end{corollary}

\begin{proof}
In light of Proposition \ref{prop:internalproduct} it suffices to notice that any $(g,h)\in\Gg\Join\Hh$ can be rewritten uniquely as
\begin{equation}\label{prod_reverse_ZS}
 (g,h) = (b(g),h|_{h\inv\cdot g\inv})(h\inv|_{g\inv}\cdot g,r(h)).
\end{equation}
That \eqref{prod_reverse_ZS} holds is a straightforward computation involving 
several applications of Lemma \ref{lem:actionrestrictioninverses} 
\eqref{inv_act} and \eqref{inv_rest}.
For the uniqueness, suppose $(g,h)=(l(h'),h')(g',t(g'))$ for some $g' \in \Gg$ and $h' \in \Hh$. Then $g=h'\cdot g'$ and $h=h'|_{g'}$. Substituting these equations into \eqref{prod_reverse_ZS} and applying Lemma \ref{lem:actionrestrictioninverses} shows that $h'=h|_{h\inv\cdot g\inv}$ and $g'=h\inv|_{g\inv}\cdot g$. Thus the decomposition in \eqref{prod_reverse_ZS} is unique.
\end{proof}

If $\Gg$ and $\Hh$ are topological groupoids with homeomorphic unit spaces, then after endowing $\Gg\restimes{t}{l}\Hh$ with the relative product topology of $\Gg \times \Hh$ it is natural to ask whether $\Gg\Join \Hh$ is a topological groupoid. It is easy to check that this is true if and only if the action and restriction maps are continuous, where $\Hh \restimes{r}{b}\Gg$ has the relative product topology.

A similar question may be asked when $\Gg$ and $\Hh$ are \'etale.

\begin{proposition}\label{etaletoetaleZS}
A Zappa-Sz\'ep product groupoid $\Gg \Join \Hh$ endowed with the relative product topology of $\Gg \times \Hh$ is \'etale if and only if both $\Gg$ and $\Hh$ are \'etale and the action and restriction maps are continuous.
\end{proposition}

\begin{proof}
Since $\Gg$ and $\Hh$ are both isomorphic to subgroupoids of $\Gg\Join \Hh$, assuming $\Gg\Join \Hh$ is \'etale immediately implies $\Gg$ and $\Hh$ are \'etale.

For the reverse implication, suppose $\Gg$ and $\Hh$ are \'etale. We must show that $(g,h) \mapsto b(g)$ is a local homeomorphism. To this end, fix $(g,h) \in \Gg\Join\Hh$. Using that $\Gg$ and $\Hh$ are \'etale we can find open subsets
 \[
  U,V \subset \Gg, \ W \subset \Hh
 \]
with $g\in U\cap V$ and $h\in W$ such that $b|_U, t|_V$ and $l|_W$ are all homeomorphisms. Define
\[
 X := ((U \cap V) \times W) \cap (\Gg \restimes{t}{l} \Hh) \subset \Gg \Join \Hh.
\]
Then $X$ is open, $(g,h) \in X$ and for any $(g',h'), (g'',h'') \in X$ we have
\begin{align*}
 b(g') = b(g'') &\Longleftrightarrow g' = g'' \mbox{ since } g', g'' \in U \\
 &\Longleftrightarrow t(g') = t(g'') \mbox{ since } g', g'' \in V \\
 &\Longleftrightarrow l(h') = l(h'') \mbox{ since } (g',h'), (g'',h'') \in \Gg \restimes{t}{l} \Hh \\
 &\Longleftrightarrow h' = h'' \mbox{ since } h', h'' \in W.
\end{align*}
We see that the range map $(g,h) \mapsto b(g)$ is a homeomorphism on $X$ as 
required.
\end{proof}

\begin{corollary}
Suppose $\Gg, \Hh$ and $\Kk$ are \'etale groupoids such that $\Gg$ and $\Hh$ 
are subgroupoids of $\Kk$ and $\Kk \cong \Gg \Join \Hh$, as in Proposition 
\ref{prop:internalproduct}. Then $\Kk \cong\Gg\Join\Hh$ is an isomorphism of 
topological groupoids and $\Kk$ is \'{e}tale.
\end{corollary}

\begin{proof}
We know from Proposition \ref{prop:internalproduct} that there is a bijective homomorphism $\Gg \Join \Hh \to \Kk$ satisfying $(g,h) \mapsto gh$. Since multiplication is continuous, to see that this map is a homeomorphism it suffices to show it is open. This follows immediately from the fact that $\Gg\Join\Hh$ is \'etale.
\end{proof}

\begin{remark}
A subset $U$ of an \'etale groupoid $\Gg$ is called a \emph{bisection} if both the range and source maps restricted to $U$ are injective. The collection $\Bis(\Gg )$ of all open bisections in $\Gg$ is an inverse semigroup under composition $UV=\{gh:(g,h)\in U\times V\cap \Gg^{(2)}\}$. If we identify a bisection $U$ with the homeomorphism $r(U) \to s(U)$ on $\Gg^{(0)}$ satisfying $r(g) \mapsto s(g)$, then it is easily checked that $\Bis(\Gg)$ is a pseudogroup of homeomorphisms of the topological space $\Gg^{(0)}$, in the sense of \cite[Section~3]{Ren2}.

The Zappa-Sz\'ep product of inverse semigroups (and semigroups more broadly) are studied in \cite{Waz}. It is natural to examine the existence of a Zappa-Sz\'ep product of $B(\Gg)$ and $B(\Hh)$, and whether there is a connection to the collection of bisections of $\Gg\Join\Hh$. We considered these problems, and there seems to be no obvious definitions for the action and restriction maps. In particular, there is no reason that the sets
\[
V\cdot U := \{ h\cdot g : h\in V, g\in U, r(h) = b(g) \}\quad\text{and}\quad V|_U := \{ h|_g : h\in V, g\in U, r(h) = b(g) \},
\]
for $U\in B(\Gg), V\in B(\Hh)$, are open bisections in $B(G)$ and $B(\Hh)$, respectively.   
\end{remark}

\section{The $C^*$-algebra of a Zappa-Sz\'ep groupoid}\label{sec: C* section}

In this section we prove the main result of this paper, which says that the 
groupoid $C^*$-algebra of the Zappa-Sz\'ep product of two 
groupoids $\Gg$ and $\Hh$ is a $C^*$-blend of the two groupoid $C^*$-algebras 
$C^*(\Gg)$ and $C^*(\Hh)$. Before we state 
this result, we briefly recall the construction of groupoid $C^*$-algebras, 
and the 
formal definition of a $C^*$-blend from \cite{Exel2014}.

For $\Gg$ a locally compact 
Hausdorff \'etale groupoid with range and source maps $r$ and $s$, define a multiplication and involution on 
$C_c(\Gg)$ by
\[
\xi * \eta (g) = \sum_{g_1g_2 = g} \xi(g_1)\eta(g_2) \quad\text{ and 
}\quad\xi^*(g) = \overline{\xi(g\inv)}.
\]
With these operations, pointwise scalar multiplication and addition, and 
$*$-algebra norm given by
\[
 \|\xi\|_I = \sup_{u\in\Gg^{(0)}} \max \left\{ \sum_{r(g) = u} |\xi(g)|, 
 \sum_{s(g) = u} |\xi(g)| \right\},
\]
$C_c(\Gg)$ becomes a normed $*$-algebra. This norm, called the $I$-norm, is 
typically not a $C^*$-norm. However, there is a $C^*$-norm on $C^*(\Gg)$ given 
by 
\[
 \|\xi\| = \sup\{ \|\pi(\xi)\| : \pi \textrm{ is a $\|\cdot\|_I$-bounded 
 $*$-representation of } C_c(\Gg) \},
\]
and the completion of $C_c(\Gg)$ under $\|\cdot\|$ is called the \emph{full 
groupoid $C^*$-algebra of }$\Gg$. 

There is also a reduced groupoid $C^*$-algebra. For each $u\in\Gg^{(0)}$ there 
is an $I$-norm-bounded representation $\Ind_u$ of $C_c(\Gg)$ on 
$\ell^2(s^{-1}(u)$) given by 
$\Ind_u(f)\delta_g=\sum_{r(h)=r(g)}f(h^{-1}g)\delta_h$. The reduced norm is 
given by $\|f\|_r=\sup_{u\in \Gg^{(0)}}\|\Ind_u(f)\|$. The completion of 
$C_c(\Gg)$ under $\|\cdot\|_r$ is called the {\em reduced groupoid 
$C^*$-algebra of $\Gg$}. For more details of these constructions we 
refer the reader to Renault's original treatment \cite{Ren}. 

We now recall Exel's notion of a $C^*$-blend from \cite{Exel2014}. 

\begin{definition}\label{def: C* blend}
For 
$C^*$-algebras $A$ and $B$ we denote by $A\otimes_\C B$ the algebraic tensor 
product. Given a $C^*$-algebra $X$ and $*$-homomorphisms 
\[
 i : A \to M(X)\quad\mbox{ and }\quad j: B \to M(X),
\]
the bilinear maps $(a,b) \mapsto i(a)j(b)$ and $(b,a) \mapsto j(b)i(a)$ 
induce linear maps 
\[
 i \circledast j : A \otimes_\C B \to M(X)\quad\text{ satisfying }\quad a\otimes b \mapsto i(a) j(b)
\]
and
\[
 j \circledast i : B\otimes_\C A \to M(X)\quad\text{ satisfying }\quad b\otimes a \mapsto j(b)i(a).
\]
A {\em $C^*$-blend} is a quintuple $(A,B,i,j,X)$, consisting of: $C^*$-algebras 
$A$, $B$, and $X$; and $*$-homomorphisms $i$ and $j$ as above, with the property that the 
range of $i\circledast j$ is contained and dense in $X$ (or, equivalently, $\range(j\circledast i)=(\range(i\circledast j))^*$ is contained and dense in $X$). 
\end{definition}

We can now state our main theorem.

\begin{theorem} \label{conj:pleasebetrue}
Suppose $\Gg\Join\Hh$ is a locally compact Hausdorff \'etale Zappa-Sz\'ep product groupoid. The maps $i : C_c(\Gg) \to C^*(\Gg\Join \Hh)$ and $j : C_c(\Hh) \to C^*(\Gg\Join \Hh)$, given by 
\[
i(\xi)(g,h) = \delta_{h,t(g)}\xi(g)\quad\text{and}\quad j(\eta)(g,h) = \delta_{g,l(h)}\eta(h),
\]
extend to $*$-homomorphisms $i:C^*(\Gg)\to C^*(\Gg\Join \Hh)$ and 
$j:C^*(\Hh)\to C^*(\Gg\Join \Hh)$, and
the quintuple $(C^*(\Gg),C^*(\Hh), i,j, C^*(\Gg\Join \Hh))$ is a $C^*$-blend.
\end{theorem}

\begin{remark}
Notice that the range of $i$ and $j$ are in $C^*(\Gg \Join \Hh)$, rather than the multiplier algebra, since $\Gg \Join \Hh$ is \'{e}tale by Proposition \ref{etaletoetaleZS}.
\end{remark}

To prove this result we need a lemma about {\em slices}, which are precompact 
open subsets of a groupoid on which the range and source maps are 
bijective. We think this lemma is well known, but we could not find a proof, 
so we include one here. Note that $\|\cdot\|_\infty$ denotes the usual supremum norm on functions.

\begin{lemma} \label{lem:groupoidalgebranorms}
Let $\Gg$ be a locally compact Hausdorff \'etale groupoid. Let $\xi\in 
C_c(\Gg) \subset C^*(\Gg)$ such that $\supp(\xi)$ is a slice. Then 
\[
\|\xi\| = \|\xi\|_r = \|\xi\|_I = \|\xi\|_\infty.
\]
\end{lemma}

\begin{proof}
We first show that $\|\xi\|_I = \|\xi\|_\infty$.  For any unit $u\in\Gu$ we 
have
\[
 |r\inv(u) \cap \supp(\xi)|, |s\inv(u) \cap \supp(\xi)| \leq 1
\]
since $r,s$ are local homeomorphisms on $\supp(\xi)$. Therefore
\[
 \|\xi\|_I = \sup_{u\in\Gg^{(0)}} \max \left\{ \sum_{r(g) = u} |\xi(g)|, 
  \sum_{s(g) = u} |\xi(g)| \right\}= \sup_{g\in\Gg} |\xi(g)|= \|\xi\|_\infty.
\]
To see that $\|\xi\|_{r} = \|\xi\|_\infty$, first fix $g\in\Gg$. We have
\[
 \xi^* \xi (g) = \sum_{hk=g} \xi^*(h)\xi(k)= \sum_{hk=g} 
 \overline{\xi(h\inv)}\xi(k).
\]
If $\xi^* \xi (g) \neq 0$, there exists $h,k\in\Gg$ with $hk=g$ and 
$\overline{\xi(h\inv)}\xi(k) \neq 0$. Then $h\inv, k \in \supp(\xi)$, and 
since $r|_{\supp(\xi)}$ is injective, we have 
\[
g=hk \Longrightarrow  r(h\inv) = r(k) \Longrightarrow h\inv = k 
\Longrightarrow g\in \Gg^{(0)}.
\]
Hence $\xi^*\xi \in C_0(\Gg^{(0)})$. Now, fix $u\in\Gu$ and 
$\delta_g\in\ell^2(\Gg_u)$, where $\Gg_u:=\{g \in \Gg : g^{-1}g=u\}$. Since $\supp(\xi^*\xi) \subset 
 \Gg^{(0)}$ we have
\[
 \Ind_u(\xi^*\xi)\delta_g =\sum_{r(h)=r(g)}\xi^*\xi(h^{-1}g)\delta_h
 = \xi^*\xi(g\inv g)\delta_g= \xi^*\xi(u)\delta_g,
\]
and so $\|\Ind_u(\xi^*\xi)\| = |\xi^*\xi(u)|$. Hence
\[
 \|\xi\|_r^2 = \|\xi^*\xi\|_r= \sup_{u\in\Gu} \|\Ind_u(\xi^*\xi)\|= 
 \sup_{u\in\Gu} |\xi^*\xi(u)|= \|\xi^*\xi\|_\infty = \|\xi\|_\infty^2.
\]
Finally, $\|\xi\|_r \leq \|\xi\| \leq \|\xi\|_I = \|\xi\|_r$, so we have shown 
all the required equalities.
\end{proof}

\begin{proof}[Proof of Theorem \ref{conj:pleasebetrue}] 
Fix $\xi \in C_c(\Gg)$. We claim that $i(\xi) \in C_c(\Gg\Join\Hh)$. To
see that $i(\xi)$ is continuous, fix an open subset $V\subseteq \C$. If 
$0\not\in V$, then
\[
i(\xi)\inv(V) = \{(g,t(g)) : \xi(g) \in V \} = (\xi\inv (V) \times \Hu )\cap 
(\Gg\restimes{t}{l}\Hh).
\] 
If $0 \in V$, then
\begin{align*}
 i(\xi)\inv(V) &= \{(g,t(g)) : \xi(g) \in V \} \cup \{(g,h) \in \Gg\restimes{t}{l} \Hh : h \in \Hh \setminus \Hu \} \\
 & = ((\xi\inv (V) \times \Hu) \cup (\Gg \times \Hh\setminus\Hu)) \cap 
 (\Gg\restimes{t}{l} \Hh).
\end{align*}
Since $\xi$ is continuous, we have $\xi\inv(V)$ open, and since $\Hh$ is 
Hausdorff and \'etale, both $\Hu$ and $\Hh\setminus\Hu$ are open. So in either 
case, $i(\xi)\inv(V)$ is open in the relative product topology, and 
$i(\xi)$ is continuous. The support of $i(\xi)$ is the set $\supp (\xi) 
\restimes{t}{l} 
\Hu$, which is homeomorphic to $\supp(\xi)$ via $(g,t(g)) \mapsto g$. Since  
$\supp(\xi)$ is compact, we have $i(\xi) \in C_c(\Gg\Join\Hh)$, as claimed. A 
symmetric argument 
using that $\Gg$ is Hausdorff and \'etale shows that $j(\eta) \in 
C_c(\Gg\Join\Hh)$ for any $\eta\in C_c(\Hh)$.

We extend $i$ and $j$ to $*$-homomorphisms $C^*(\Gg)\to C^*(\Gg\Join \Hh)$ and 
$C^*(\Hh)\to C^*(\Gg\Join \Hh)$, respectively, and
we now claim that $(C^*(\Gg),C^*(\Hh), i,j, C^*(\Gg\Join \Hh))$ is a 
$C^*$-blend. Firstly, for each $\xi\in C_c(\Gg)$ and $\eta\in C_c(\Hh)$ we have
\[
 i\circledast j (\xi\otimes\eta)(g,h) = \xi(g)\eta(h),
\]
from which we see that $ i\circledast j (\xi\otimes\eta)$ is continuous. 
We also have 
\[
\supp(i\circledast j(\xi\otimes\eta)) = (\supp(\xi)\times\supp(\eta)) \cap 
(\Gg\restimes{t}{l}\Hh),
\]
which is compact. So the image of $i\circledast j$ is contained in 
$C_c(\Gg\Join \Hh)$. To complete the proof we need to show that this image is 
dense in $C^*(\Gg\Join \Hh)$.

For an arbitrary function $\theta\in C_c(\Gg\Join \Hh)$ we can cover the 
support
by a finite number of precompact open bisections $\{U_k : 1\leq k \leq n \}$. 
If $\{ \pi_k : 1\leq k \leq n \}$ is a partition of unity for $\supp(\theta)$ 
with $\supp(\pi_k) \subset U_k$, then $\theta = \sum_{k=1}^n \theta \pi_k$, 
where $\supp (\theta \pi_k)$ is a precompact open bisection. Since we know 
from Lemma~\ref{lem:groupoidalgebranorms} that $\|\theta \pi_k\|_\infty = 
\|\theta \pi_k\|_I$, it suffices to show that the image of $i\circledast j$ is 
\emph{uniformly} dense in $C_c(\Gg\Join \Hh)$. To this end, fix $(g,h) \neq 
(g',h')$. By the Stone-Weierstrass theorem for locally compact spaces, it is 
enough to find $\xi\in C_c(\Gg)$ and $\eta\in C_c(\Hh)$ with $i\circledast 
j(\xi\otimes\eta)(g,h) =1$ and $i\circledast j (\xi\otimes\eta)(g',h') = 0$. 
Without loss of generality assume $g\neq g'$. Fix $\xi\in C_c(\Gg)$ with 
$\xi(g)=1, \xi(g') = 0$ and $\eta\in C_c(\Hh)$ with $\eta(h) = 1$. Then
\[
 i\circledast j(\xi\otimes\eta)(g,h) = \xi(g)\eta(h) = 1\quad\text{and}\quad 
 i\circledast j(\xi\otimes\eta)(g',h') = \xi(g')\eta(h') = 0,
\]
as required.
\end{proof}

\begin{remark}\label{ref: chars of blends}
If $G$ and $H$ are subgroups 
of a 
group $K$, then $GH=\{gh:g\in G,h\in H\}$ is a group if and only if $GH=HG$. 
Moreover, $GH$ is isomorphic to $G\Join H$ if and only if $G\cap H=\{e\}$ (see \cite[Satz~6]{Redei}). There is a 
similar characterisation for the product of subsets of a $C^*$-algebra to be a 
$C^*$-blend. Suppose $A$ and $B$ are $C^*$-subalgebras of a $C^*$-algebra $C$, and 
denote by $AB$ the 
set $\overline{\text{span}}\{ ab : a\in A, b\in B \}$. Then the following 
are equivalent:
\begin{enumerate}
\item $AB = BA$,
\item $AB$ is a $C^*$-algebra,
\item there exist $C^*$-homomorphisms $i : A \to M(AB)$ and $j: B \to M(AB)$ 
such that $(A,B,i,j,AB)$ is a $C^*$-blend.
\end{enumerate} 
For $(2)\implies (3)$ we use maps $i : A\to M(AB)$ and $j : B\to M(AB)$ given 
by $i(a)x = ax \text{ and } j(b)x = bx$. Implications $(3)\implies (2)$ 
and $(1)\Longleftrightarrow (2)$ are straightforward exercises.
\end{remark}

\section{Examples}\label{sec:examples}

In our final section we examine several examples of Zappa-Sz\'{e}p product 
groupoids and their $C^*$-algebras.

\subsection{$*$-commuting endomorphisms} \label{sec:starcommutingendos} 

In this section we show that every pair of $*$-commuting endomorphisms of a 
topological space gives rise to 
a Zappa-Sz\'ep product of Deaconu-Renault groupoids (see Example~\ref{eg: DR 
groupoid}). 

Recall from \cite{AR} that 
a pair of commuting endomorphisms $S$ and $T$ of a topological space $X$ are 
said to 
{\em $*$-commute} if, for every 
$x,y \in X$ with $Tx = Sy$, there exists a unique $z \in X$ with $Sz = x$ and 
$Tz = y$. We call such $S$ and $T$ {\em $*$-commuting endomorphisms}.

\begin{proposition} \label{prop:starcommute}
Suppose $S$ and $T$ are $*$-commuting endomorphisms of a topological space $X$. Then there is an action $\theta$ of $\mathbb{N}^2$ on $X$ given by $\theta_{(m_1,m_2)}=S^{m_1} T^{m_2}$, and the Deaconu-Renault groupoid for this action is the internal Zappa-Sz\'ep product of the individual Deaconu-Renault groupoids for the actions of $\N$ on $X$ induced by $S$ and $T$.
\end{proposition}

Since $S$ and $T$ commute, $\theta_{(m_1,m_2)}:= S^{m_1} T^{m_2}$ gives an 
action $\theta$ of $\mathbb{N}^2$ by continuous endomorphisms of $X$. Let 
$X\rtimes_\theta \N^2$ be the corresponding Deaconu-Renault groupoid, and 
$X\rtimes_S\N$ and $X\rtimes_T\N$ be the 
Deaconu-Renault groupoids for the actions of $\N$ on $X$ induced by $S$ and $T$, respectively. Notice that $X\rtimes_S\N$ and $X\rtimes_T\N$ can be viewed as subgroupoids of $X \rtimes_\theta \N^2$ via
\begin{align*}
X \rtimes_S \N &\cong \{ (x,m-n,y) : m,n\in\N\times\{0\} \},\text{ and}\\
X \rtimes_T \N &\cong \{ (x,m-n,y) : m,n\in\{0\} \times \N \}.
\end{align*}
So to prove Proposition~\ref{prop:starcommute} we need to show that
\[
 X\rtimes_\theta \N^2 \cong (X\rtimes_S\N) \Join (X\rtimes_T\N).
\]
\begin{proof}[Proof of Proposition~\ref{prop:starcommute}]
We aim to use Proposition~\ref{prop:internalproduct}. Fix $(x,m-n,y)\in 
X \rtimes_\theta \N^2$. Write $m = (m_1,m_2)$ and $n = (n_1,n_2)$. By 
definition of $X \rtimes_\theta \N^2$ we have
\[
 S^{n_1} T^{n_2} y = T^{m_2} S^{m_1} x
\]
Since $S$ and $T$ $*$-commute, the maps $S^{m_1}$ and $T^{m_2}$ also 
$*$-commute. Therefore, there is a unique $z\in X$ such that $S^{n_1}z = 
S^{m_1}x$ and $T^{m_2}z = T^{n_2}y$. This information is summarised in the 
following diagram:
\[
\begin{tikzpicture}
\node (y) at (0,4) {$y$};
\node (z) at (3,4) {$z$};
\node (x) at (6,4) {$x$};
\node (Ty) at (0,2) {$T^{n_2}y$};
\node (Sx) at (6,2) {$S^{m_1}x$};
\node (eq) at (3,0) {$S^{n_1}T^{n_2}y = T^{m_2}S^{m_1}x$};

\draw [|->] (y) -- node[left] {$T^{n_2}$} (Ty);
\draw [|->] (z) -- node[left] {$T^{m_2}$} (Ty);
\draw [|->] (x) -- node[right] {$S^{m_1}$} (Sx);
\draw [|->] (z) -- node[right] {\,\,\,$S^{n_1}$} (Sx);
\draw [|->] (Ty) -- node[right] {\,\,\,$S^{n_1}$} (eq);
\draw [|->] (Sx) -- node[above] {$T^{m_2}$} (eq);

\end{tikzpicture}
\]
Hence, we have elements
\[
 (x,(m_1,0)-(n_1,0),z) \in X \rtimes_S \N \subset X \rtimes_\theta \N^2
\]
and
\[
 (z,(0,m_2)-(0,n_2),y) \in X\rtimes_T \N \subset X \rtimes_\theta \N^2
\]
with $(x,(m_1,0)-(n_1,0),z)(z,(0,m_2)-(0,n_2),y) = (x,m-n,y)$. Since 
$z$ was unique, this decomposition is also unique and so Proposition 
\ref{prop:internalproduct} provides us with the desired isomorphism.
\end{proof}

\begin{remark}\label{rem: DR blend}
Applying Theorem~\ref{conj:pleasebetrue} in this setting gives a $C^*$-blend 
\[
(C^*(X\rtimes_S\N),C^*(X\rtimes_T\N), i,j, C^*(X\rtimes_\theta \N^2)).
\]
\end{remark}

\subsection{$1$-coaligned $2$-graphs}

We know from \cite[Defintion~2.1]{MW2012} (also see \cite{W}) that examples of $*$-commuting maps come from the shift map on certain $2$-graphs. We recall the details. 

We view the monoid $\N^2$ as a category with one object 
in the usual way. We write $e_1 = (1,0)$ and $e_2 = (0,1)$ for the canonical 
generators. Recall from \cite{KP2000} that a \emph{$2$-graph} is a small 
category 
$\Lambda$ equipped with a degree functor $d: \Lambda \to \N^2$ which satisfies 
the factorisation property, in the sense that whenever $\lambda\in\Lambda$ and 
$m,n\in\N^2$ satisfy $d(\lambda) = m+n$, 
there are unique elements $\mu,\nu\in\Lambda$ satisfying $d(\mu)=m$, 
$d(\nu)=n$ and $\lambda = \mu\nu$. The objects of $\Lambda$ can be identified 
with $\Lambda^0:=d^{-1}(0)$. The codomain and domain maps are denoted by $r$ 
and $s$, and are called the range and source maps. A $2$-graph is called {\em 
row-finite and with no sources} if for every $v\in\Lambda^0$ and $n\in\N^2$ 
the set $\{\lambda\in\Lambda:d(\lambda)=n,\, r(\lambda)=v\}$ is nonempty and 
finite. 

Let $\Omega_2$ be the category with objects $\N^2$, morphisms 
$\{(m,n) : m,n\in\N^2, m\leq n\}$ where $\N^2$ has the usual partial order, 
and range and source maps $r(m,n) = m, s(m,n) = n$. With the degree functor $d(m,n) = 
n-m$, $\Omega_2$ is a $2$-graph. If $\Lambda$ is a $2$-graph, an 
\emph{infinite path} in $\Lambda$ is a functor $x : \Omega_2 \to 
\Lambda$. We write $\Lambda^\infty$ for the space of infinite paths. If 
$\Lambda$ is row-finite and with no sources, then $\Lambda^\infty$ equipped 
with the topology generated by cylinder sets $Z(\lambda) := 
\{x\in\Lambda^\infty: x(0,d(\lambda)) = \lambda \}$ is a totally disconnected 
locally compact 
Hausdorff space. For each $p\in\N^2$ consider the shift map $\sigma^p : 
\Lambda^\infty 
\to \Lambda^\infty$ given by $\sigma^p(x)(m,n) = x(m+p,n+p)$. Each $\sigma^p$ 
is a local homeomorphism. If in addition $\Lambda$ has no 
sinks, in the sense that for the each $v\in\Lambda^0$ and $n\in\N^2$ the set 
$\{\lambda\in\Lambda:d(\lambda)=n,\, s(\lambda)=v\}$ is nonempty, then each 
$\sigma^p$ is also surjective. So for $\Lambda$ a $2$-graph which is 
row-finite and with no sinks or sources, the map $k\mapsto\sigma^k$ 
determines an action of $\N^2$ by endomorphisms 
$\Lambda^\infty$. Let $\Gg_\Lambda = \Lambda^\infty \rtimes \N^2$ be the 
associated Deaconu-Renault groupoid.

\begin{definition}{\cite[Defintion~2.1]{MW2012}}
A $2$-graph $\Lambda$ is \textit{$1$-coaligned} if for every pair $(e^1, e^2) 
\in \Lambda^{e_1}  \restimes{s}{s} \Lambda^{e_2}$ there exists a unique pair $(f^1, 
f^2) \in \Lambda^{e_1} \restimes{r}{r} \Lambda^{e_2}$ such that $f^1e^2 = f^2e^1$.
\end{definition}

\noindent
A large class of examples of $1$-coaligned $2$-graphs are provided in 
\cite[Theorem 3.1]{MW2012}. The connection between $1$-coaligned $2$-graphs 
and $*$-commuting endomorphisms comes from the following result (which applies 
to more general $k$-graphs, but we state only for $2$-graphs).

\begin{theorem}{\cite[Corollary 2.4]{MW2012}}
If $\Lambda$ is a 1-coaligned row-finite $2$-graph with no sinks or sources, 
then for each $i\not=j$, $\sigma^{e_i}$ and $\sigma^{e_j}$ are $*$-commuting 
surjective local homeomorphisms.
\end{theorem} 

Using our results we can now decompose both the graph groupoid of a 
$1$-coaligned $2$-graph, and its groupoid $C^*$-algebra, the graph algebra. 
We direct the reader to \cite{Raeburn} for an account of directed graphs, 
$k$-graphs, and their $C^*$-algebras. For a $2$-graph $\Lambda$ we define the 
{\em blue graph} $B_\Lambda$ and the {\em red graph} $R_\Lambda$ to be the 
directed graphs
\[
B_\Lambda=(B_\Lambda^0:=\Lambda^0,B_\Lambda^1:=\Lambda^{e_1}, r|_{\Lambda^{e_1}},s|_{\Lambda^{e_1}})\quad\text{and}\quad R_\Lambda=(R_\Lambda^0:=\Lambda^0,R_\Lambda^1:=\Lambda^{e_2}, r|_{\Lambda^{e_2}},s|_{\Lambda^{e_2}}).
\] 

\begin{theorem}\label{thm: 2-blue-red graph}
For every 1-coaligned row-finite $2$-graph $\Lambda$ with no sinks or sources 
we have
\[
\Gg_\Lambda \cong (\Lambda^\infty \rtimes_{\sigma^{e_1}} \N) \Join 
 (\Lambda^\infty \rtimes_{\sigma^{e_2}} \N).
\]
Moreover, there are $*$-homomorphisms $i:C^*(B_{\Lambda})\to C^*(\Lambda)$ and 
$j:C^*(R_\Lambda)\to C^*(\Lambda)$ which make 
$(C^*(B_\Lambda),C^*(R_\Lambda),i,j,C^*(\Lambda))$ a $C^*$-blend.
\end{theorem}

This result follows from Theorem~\ref{conj:pleasebetrue} once the 
isomorphisms 
$C^*(B_\Lambda)\cong C^*(\Lambda^\infty \rtimes_{\sigma^{e_1}} \N)$ and 
$C^*(R_\Lambda)\cong C^*(\Lambda^\infty \rtimes_{\sigma^{e_2}} \N)$ are 
established; this is an exercise in finding appropriate Cuntz-Krieger families 
in $C^*(\Lambda^\infty \rtimes_{\sigma^{e_1}} \N)$ and $C^*(\Lambda^\infty 
\rtimes_{\sigma^{e_2}} \N)$, and applying the gauge-invariant uniqueness 
theorem. We leave the details to the reader.

\subsection{Skew product groupoids}\label{sec:skewproductgroupoid}

Fix an \'etale groupoid $\Gg$, a discrete group $A$ and a continuous 
homomorphism $c : \Gg \to A$. Recall from Example~\ref{ex:skewproductgroupoid} 
the construction of the skew product groupoid $\Gg(c)$; this groupoid is also 
\'{e}tale because $A$ is discrete. 

The formula $\beta\cdot (g,\alpha) := 
(g,\alpha\beta\inv)$ defines a left action of $A$ on the space 
$\Gg(c)$. For
a composable pair $((g,\alpha),(h,\alpha c(g)))\in\Gg(c)^{(2)}$, this 
action 
satisfies
\begin{align*}
 \beta \cdot (g,\alpha)(h,\alpha c(g)) = \beta\cdot (gh,\alpha)
 &= (gh,\alpha \beta\inv) \\
 &= (g,\alpha \beta\inv)(h,\alpha \beta\inv c(g)) \\
 &= \big(\beta\cdot (g,\alpha) \big)(h,\alpha c(g)c(g)\inv \beta\inv 
 c(g)) \\
 &= \big(\beta\cdot (g,\alpha) \big)\big( c(g)\inv \beta c(g) \cdot 
 (h,\alpha 
 c(g)) \big).
\end{align*}
This identity is suggestive of a Zappa-Sz\'ep product structure on $\Gg 
(c)\times A$ 
with restriction given by $\beta|_{(g,\alpha)} := c(g)\inv \beta c(g)$. 
The 
next result 
says that is indeed the case, although 
we have to be careful with the unit spaces, which makes the details a 
little 
more complicated.

\begin{proposition}\label{prop: skew prod ZSP}
Fix an \'etale groupoid $\Gg$, a discrete group $A$ and a continuous 
homomorphism $c : \Gg \to A$. There is a left action of $A$ on the 
space 
$\Gg^{(0)} 
\times A$ given by $\beta\cdot (u,\alpha) := (u,\alpha\beta\inv)$. If 
$\Hh:= A 
\ltimes 
(\Gg^{(0)} \times A)$ denotes the corresponding transformation 
groupoid, with 
range and source maps denoted $l$ and $r$, then the maps $\cdot:\Hh 
\restimes{r}{b} \Gg(c)\to \Gg(c)$ and $|:\Hh \restimes{r}{b} \Gg(c)\to 
\Hh$ 
given by
\[
(\beta,(gg\inv,\alpha))\cdot(g,\alpha) := (g,\alpha \beta\inv) 
\quad\text{and}\quad(\beta,(gg\inv,\alpha))|_{(g,\alpha)} := (c(g)\inv 
\beta 
c(g),(g\inv g,\alpha c(g)))
\]
satisfy (ZS1--9), and hence induce a Zappa-Sz\'ep product groupoid 
$\Gg(c)\Join\Hh$.
\end{proposition}

The proof of this result is nothing more than a checklist of what it 
takes for 
$\beta\cdot (u,\alpha) := (u,\alpha\beta\inv)$ to give an action, and 
the 
axioms (ZS1--9). As 
each calculation is routine, we leave the details to the reader. Notice 
that 
an arbitrary element of $\Gg(c)\Join\Hh$ has the form 
$((g,\alpha),(\beta,(g\inv g,\alpha c(g)\beta)))$, and is completely 
determined by the elements $(g,\alpha)\in\Gg(c)$ and $\beta\in A$. 
So as a space it is homeomorphic to $\Gg(c) \times A$ and in some sense 
should be considered as the Zappa-Sz\'ep product of the groupoid 
$\Gg(c)$ with 
the group $A$.

The $C^*$-algebras of skew product groupoids are well studied in 
\cite{Ren,KQR2001}, and we use the notation of \cite{Ren}. We know 
from  
\cite{KQR2001} that $c$ induces a coaction $\delta_c : C^*(\Gg) \to 
C^*(\Gg) \otimes C^*(A)$ satisfying $\delta_c(\xi) = \xi \otimes U_b$
whenever $\xi\in C_c(\Gg)$ satisfies $\supp\xi \subset c\inv(\{b\})$. 
In 
\cite[Theorem~4.3]{KQR2001} the authors show that $C^*(\Gg(c))$ is 
isomorphic 
to the coaction crossed product $C^*(\Gg) 
\rtimes_{\delta_c} A$. The canonical left-action $\beta \cdot 
(g,\alpha) 
\mapsto (g,\beta\alpha)$ commutes with right multiplication in $\Gg(c)$ 
and 
hence induces an action 
$\gamma:A\to\Aut C^*(\Gg(c))$ characterised by 
$\gamma_\beta(\xi)(g,\alpha) = 
\xi(\beta\inv\cdot(g,\alpha))$, for $\xi\in C_c(\Gg(c))$. In 
\cite[Corollary~4.5]{KQR2001} it is shown that $\gamma$ is dual to the 
coaction 
$\delta_c$, so that $C^*(\Gg(c)) \rtimes_\gamma A \cong C^*(\Gg) 
\otimes 
\Kk(\ell^2(A))$.

The alternative left action $\beta\cdot (g,\alpha) = 
(g,\alpha\beta\inv)$ that we used to build the Zappa-Sz\'ep product 
does {\em 
not} commute with right multiplication in $\Gg(c)$, and hence does not
induce an action of $A$ on $C^*(\Gg(c))$. We do not therefore expect a 
crossed-product description of $C^*(\Gg(c)\Join\Hh)$. 
Theorem~\ref{conj:pleasebetrue} does apply and says that 
$C^*(\Gg(c)\Join\Hh)$ 
is the blend of $C^*(\Gg(c))$ and $C^*(\Hh)$.   

We can also say a little more about the Zappa-Sz\'ep product groupoid 
$\Gg(c)\Join\Hh$. There is a right action of $\Gg$ by automorphisms 
of $A$ given by $\alpha\cdot g = c(g)\inv \alpha c(g)$, from which we can form 
the semidirect product groupoid $\Gg\ltimes A$. Pairs $((g,\alpha),(h,\beta))$ 
are composable in $\Gg\ltimes A$ if $(g,h)\in \Gg^{(2)}$, and composition and 
inversion are given by $(g,\alpha)(h,\beta) = 
(gh,c(h)\inv \alpha c(h) \beta)$ and $(g,\alpha)\inv = 
(c(g)\alpha\inv c(g)\inv,\alpha\inv)$. One can check that the map 
$\widetilde{c} : \Gg\ltimes A \to A$ satisfying $\widetilde{c}(g,\alpha) = 
c(g)\alpha$ is a continuous groupoid homomorphism, and that the map 
$((g,\alpha),\beta) \mapsto 
((g,\alpha),(\beta,(g\inv g,\alpha c(g)\beta)))$ is a 
groupoid isomorphism of $(\Gg\ltimes A)(\widetilde{c})$ onto $\Gg(c)\Join\Hh$. 
So despite $C^*(\Gg\Join\Hh)$ not admitting a natural crossed-product 
description, we can use the results of \cite{KQR2001} to describe it as the 
coaction crossed product $C^*(\Gg\ltimes A) \rtimes_{\delta_{\widetilde{c}}} 
A$.

\end{document}